\documentclass[11pt,reqno]{amsart}

\parskip=\smallskipamount

\newtheorem{theorem}{Theorem}[section]
\newtheorem{lemma}[theorem]{Lemma}
\newtheorem{corollary}[theorem]{Corollary}

\newtheorem{theo}{Theorem}[section]
\theoremstyle{definition}

\newtheorem{conjecture}[theorem]{Conjecture}

\newcommand{\bea}{\begin{eqnarray*}}
\newcommand{\eea}{\end{eqnarray*}}

\numberwithin{equation}{section}

\begin{document}

\centerline{On a  conjecture by Blocki-Zwonek}

\centerline{by}

\centerline{John Erik Forn\ae ss}

\centerline{Department of Mathematics, NTNU Trondheim, Norway.}

Abstract. This paper gives a counterexample to a conjecture by Blocki and Zwonek on the
area of sublevel sets of the Green's function.

\section{Introduction}

In [1] Blocki and Zwonek make the following conjecture:

\begin{conjecture}
If $\Omega$ is a pseudoconvex domain in $\mathbb C^n,$ then the function
$$
t\rightarrow \log \lambda(\{G_{\Omega,w}<t\})
$$
is convex. 
\end{conjecture}

Here $w$ is any given point in $\Omega$ and $G$ denotes the pluricomplex Green function with
pole at $w.$ Also $\lambda$ denotes Lebesgue measure.

In this note we provide a counterexample in $\mathbb C.$ The example is constructed in the next section.
It is a triply connected domain where the Green function has a higher order critical point.

\section{The example}

\begin{lemma}
There exists a bounded domain $\Omega\subset \mathbb C$ and a point $q\in \Omega$ such that the Green function 
$G$ with pole at $q$ has a critical point $p\in  \Omega$ so that $G(z)-G(p)=\mathcal O(|z-p|^3)$ in a neighborhood of $p.$ The domain is triply connected with boundary consisting of three simply closed real analytic smooth curves.
\end{lemma}

\begin{proof}

We construct the domain $\Omega$ by modifying domains in steps:
$$\Omega_1,\Omega_2,\Omega_3, \Omega_4=\Omega.$$
First, we define $\Omega_1\subset \mathbb P^1.$ The point $\infty$ will be an interior point.
The constant $r=e^{-1/3}$, $0<r<1.$ 

$$
\Omega_1:=\{\eta\in \mathbb P^1; |\eta+1|>r.\}
$$
So $\Omega_1$ is a topological disc. $\eta=0\in \Omega_1.$ The complement of $\Omega_1$ is the closed disc centered at $\eta=-1$ with radius $r.$ 

Next we let 

$$
\Omega_2:=\{\tau\in \mathbb P^1; |\tau^3+1|>r\}.
$$

Again $\infty $ is in the domain. The points $-1, e^{i\pm \frac{\pi}{3}}=c_1,c_2,c_3$ belong to three
simply connected domains $D_j$ with real analytic boundary $\gamma_1, \gamma_2, \gamma_3$. The union of their (pairwise disjoint)
closures constitute the complement of $\Omega_2$ in $\mathbb P^1.$ The origin and infinity are in the interior of the domain.

We tranlate 1 unit to the right:

$$
\Omega_3:= \{\sigma\in \mathbb P^1; |(\sigma-1)^3+1|>r\}
$$

The points $0,1+e^{\pm i \frac{\pi}{3}}=c'_1,c'_2,c'_3$ belong to three simply connected domains $D_j'$ with real analytic boundary $\gamma'_1,\gamma_2',\gamma_3'$. The union of their (pairwise disjoint) closures constitute the complememt of $\Omega_3$.
The point 1 and infinity belong to the interior of $\Omega_3.$ Also $c'=0$ is contained in a neighborhood $ D_1'$
which is in the complement of $\Omega_3.$

Next we make an inversion so that the domain becomes bounded.

$$
\Omega_4=\Omega=\{w\in \mathbb P^1; |\left( \frac{1}{w} -1    \right)^3+1|>r\}.
$$

The points $\infty, \frac{1}{1+e^{\pm i \frac{\pi}{3}}}=c''_1,c_2'',c_3''$ belong to three simply connected domains $D_j''$ with real analytic boundary $\gamma_1'', \gamma_2'', \gamma_3''$. The union of their closures constitute the complement of $\Omega.$
The point 1 and 0 belong to the interior of $\Omega.$ $\Omega$ lies inside the curve $\gamma_1''$ and their
are two holes, bounded by $\gamma_2'',\gamma_3''$ respectively.

We rewrite the description of $\Omega:$

$$
\Omega = \{w\in \mathbb P^1; \log  |\left( \frac{1}{w} -1    \right)^3+1|-\log r>0\}
$$

$$
\Omega = \{w\in \mathbb P^1; \log | \frac{(1-w)^3+w^3}{w^3}          |-\log r>0\}
$$

$$
\Omega = \{w\in \mathbb P^1; \log | \frac{w^3}{(1-w)^3+w^3}         |+\log r<0\}
$$

$$
\Omega = \{w\in \mathbb P^1; \log | \frac{w^3}{1-3w+3w^2}         |+\log e^{-1/3}<0\}
$$

$$
\Omega = \{w\in \mathbb P^1; \frac{1}{3}\log | \frac{w^3}{1-3w+3w^2}         |+\frac{1}{3}(-1/3)<0\}
$$

Let $G(w)=\frac{1}{3}\log | \frac{w^3}{1-3w+3w^2}         |-\frac{1}{9}$
Then $\Omega$ is the set where this function is negative. In a neighborhood of the origin, we have $G(w)=\log|w|+\mathcal O(1).$ In the rest of $\Omega$ the function is harmonic. Hence $G(w) $ is the Green function for $\Omega$ with a pole at the origin.  

Next we consider the function $G$ in a neighborhood of the interior point $w=1:$
We get $G(1)=\frac{-1}{9}.$
We estimate $G(w)+\frac{1}{9}:$

\bea
G(w)+\frac{1}{9} & = & \frac{1}{3} \log |\frac{w^3}{w^3+(1-w)^3}|\\
& = & \frac{1}{3}\log|\frac{1}{1+\left(\frac{1-w}{w}  \right)^3}|\\
& = & \frac{1}{3}\log|1-\left(\frac{1-w}{w}  \right)^3+\mathcal O( \left(\frac{1-w}{w}  \right)^6)|\\
& = & \frac{1}{3}(Re)(\log (1-\left(\frac{1-w}{w}  \right)^3+\mathcal O( \left(\frac{1-w}{w}  \right)^6)\\
& = & \frac{1}{3}(Re)(-\left(\frac{1-w}{w}  \right)^3+\mathcal O( \left(\frac{1-w}{w}  \right)^6)\\
& = & 3(Re)(-(1-w)^3+\mathcal O(1-w)^4\\
& = & \mathcal O(|w-1|^3)\\
\eea

\end{proof}

\begin{lemma}
Let $S_\epsilon$ denote the sector $S_\epsilon=\{w=1+re^{i\theta}; 0<r<\epsilon^{1/3}/2, \frac{11\pi}{12}
<\theta<\frac{13\pi}{12}\}.$ There is an $\epsilon_0>0$ so that if
$0<\epsilon<\epsilon_0$ then 
$$
S_\epsilon\subset \{w\in \Omega; -\epsilon<G(w)-G(1)<0\}.
$$
\end{lemma}

\begin{proof}
Suppose that $w\in S_\epsilon.$ Then 

\bea
G(w)-G(1) & = & G(w)+\frac{1}{9}\\
& = & \frac{1}{3}(Re)(w-1)^3+\mathcal O(w-1)^4\\
& \Rightarrow & \\
\frac{1}{3}(Re)((w-1)^3)-Cr^4 & \leq & G(w)-G(1)\\
& \leq & \frac{1}{3}(Re)((w-1)^3)+Cr^4\\
& \Rightarrow & \\
3r^3\cos(3\theta)-Cr^4 & \leq & G(w)+1\\
& \leq & 3r^3\cos(3\theta) +Cr^4\\
2\pi + \frac{3}{4}\pi  
& \leq & 3\theta\\
& \leq & 2\pi+\frac{5}{4}\pi\\
& \Rightarrow & \\
-1 & \leq & cos (3\theta)\\
& \leq & -\sqrt{2}/2\\
-\frac{1}{3} r^3-Cr^4 & \leq & G(w)-G(1)\\
& \leq & -\frac{1}{3}\sqrt{2}/2 r^3+Cr^4\\
-\frac{\epsilon}{24}-C\frac{\epsilon^{4/3}}{16} & \leq & G(w)-G(1)\\
& \leq & -\frac{\sqrt{2}\epsilon}{48}\\
& \Rightarrow & \\
-\epsilon & < & G(w)-G(1)<0\\
\eea
\end{proof}

Define the function $s(t)$ for $-\infty<t<0.$
$$
s(t)= \lambda(\{G(w)<t\})
$$
i.e. the area of the sublevel set where $G<t.$

\begin{corollary}
For all small enough $\epsilon>0$ we have that
$$
s(-\frac{1}{9})-s(-\frac{1}{9}-\epsilon) \geq \pi \frac{\epsilon^{2/3}}{48}.
$$
\end{corollary}

\begin{proof}
We have that 
\bea
s(-\frac{1}{9})-s(-\frac{1}{9}-\epsilon) & = & \lambda(\{-\frac{1}{9}-\epsilon<G<-\frac{1}{9}\})\\
& 
\geq & \lambda(S_\epsilon)\\
& = & \pi (\frac{\epsilon^{1/3}}{2})^2 \frac{1}{12}\\
\eea

This implies that
$$
s(-\frac{1}{9}-\epsilon)\leq s(-\frac{1}{9})-\pi\frac{\epsilon^{2/3}}{48}.
$$

\end{proof}

The following is immediate.

\begin{theo}
The function is not convex.
\end{theo}

\begin{corollary}
The function $t\rightarrow \log \lambda (\{G<t\})$ is not convex.
\end{corollary}

\bibliographystyle{amsplain}

\end{document}